\documentclass[12pt, a4paper]{article}
\usepackage{amsfonts}
\usepackage{amssymb}
\usepackage{bbm}
\usepackage{amscd}
\usepackage{mathrsfs}
\usepackage{amsmath,amssymb, amsthm}
\usepackage{graphicx}
\usepackage[T1]{fontenc}
\usepackage{color}
\usepackage{authblk}
\usepackage{subcaption}
\usepackage{indentfirst,latexsym,amssymb}
\usepackage[bookmarks=false]{hyperref}
\usepackage{tikz}

\newtheorem{theorem}{Theorem}[section]

\newtheorem{lem}[theorem]{Lemma}

\newtheorem{coro}[theorem]{Corollary}

\theoremstyle{definition}

\theoremstyle{definition}
\newtheorem{remark*}{Remark}

\def\00{\mathbf{0}}

\def\Cay{\hbox{\rm Cay}}

\def\Ga{\Gamma}

\usepackage{xcolor}
\usepackage[normalem]{ulem}

\begin{document}

\title{A note on regular sets in Cayley graphs}
\author[a]{Junyang Zhang}
\author[b]{Yanhong Zhu}
\affil[a]{\small School of Mathematical Sciences, Chongqing Normal University

Chongqing 401331, People's Republic of China}
\affil[b]{\small  School of Mathematical Sciences, Liaocheng University

Liaocheng 252000, People's Republic of China}
\date{}

\openup 0.5\jot
\maketitle

\renewcommand{\thefootnote}{\fnsymbol{footnote}}
\footnotetext{E-mail address: jyzhang@cqnu.edu.cn (Junyang Zhang); zhuyanhong911@163.com (Yanhong Zhu)}

\begin{abstract}
A subset $R$ of the vertex set of a graph $\Ga$ is said to be $(\kappa,\tau)$-regular if $R$ induces a $\kappa$-regular subgraph and every vertex outside $R$ is adjacent to exactly $\tau$ vertices in $R$. In particular, if $R$ is a $(\kappa,\tau)$-regular set of some Cayley graph on a finite group $G$, then $R$ is called a $(\kappa,\tau)$-regular set of $G$. Let $H$ be a non-trivial normal subgroup of $G$, and $\kappa$ and $\tau$ a pair of integers satisfying $0\leq\kappa\leq|H|-1$, $1\leq\tau\leq|H|$ and $\gcd(2,|H|-1)\mid\kappa$. It is proved that (i) if $\tau$ is even, then $H$ is a $(\kappa,\tau)$-regular set of $G$; (ii) if $\tau$ is odd, then $H$ is a $(\kappa,\tau)$-regular set of $G$ if and only if it is a $(0,1)$-regular set of $G$.


\end{abstract}
\section{Introduction}
In the paper, all groups considered are finite groups with identity element denoted
as $1$, and all graphs considered are finite, undirected and simple. Let $R$ be a subset of the vertex set of a graph $\Ga$, and $\kappa$ and $\tau$ a pair of nonnegative integers. We call $R$ a $(\kappa,\tau)$-\emph{regular set} (or \emph{regular set} for short if there is no need to emphasize the parameters $\kappa$ and $\tau$ in the context) of $\Gamma$ if every vertex in $R$ is adjacent to exactly $\kappa$ vertices in $R$ and every vertex outside $R$ is adjacent to exactly $\tau$ vertices in $R$. In particular, we call $R$ a perfect code of $\Gamma$ if $(\kappa,\tau)=(0,1)$ and a total perfect code of $\Gamma$ if $(\kappa,\tau)=(1,1)$.
The concept of $(\kappa,\tau)$-\emph{regular set} was introduced in \cite{CR2004} and further studied in \cite{AC2013,Car2019, CR2007,CS2010}. Very recently, regular sets in Cayley Graphs was studied in \cite{WXZ2022,WXZ2022+}.

Let $G$ be a group and $X$ an inverse closed subset of $G\setminus\{1\}$. The Cayley graph $\Cay(G,X)$ on $G$ with connection set $X$ is the graph with vertex set $G$ and edge set $\{\{g,gx\}\mid g\in G, x\in X\}$. A subset $R$ of $G$ is called a $(\kappa,\tau)$-\emph{regular set} of $G$ if there is a Cayley graph $\Gamma$ on $G$ such that $R$ is a $(\kappa,\tau)$-regular set of $\Gamma$.
Regular sets of Cayley graphs are closely related to codes of groups.
Let $C$ and $Y$ be two subsets of $G$ and $\lambda$ a positive integer. If for every $g\in G$ there exist  precisely $\lambda$ pairs $(c,y)\in C\times Y$ such that $g=cy$, then $C$ is called a \emph{code} of $G$ with respect to $Y$ \cite{GL2020}.
In particular, if $\lambda=1$ and $Y$ is an inverse closed subset of $G$ containing $1$,
then $C$ is called a \emph{perfect code} of $G$ \cite{HXZ2018}.
Let $H$ be a subgroup of $G$. It is straightforward to check that $H$ is a $(0,\tau)$-regular set of $G$ if and only if $H$ is a code of $G$ with respect to some inverse closed subset of $G$. Actually, if $H$ is a $(0,\tau)$-regular set of the Cayley graph $\Cay(G,X)$, then $H$ is a code of $G$ with respect to $Y:=X\cup Z$ for any inverse closed subset $Z$ of $H$ with cardinality $\tau$; and on the other hand, if $H$ is a code of $G$ with respect to $Y$, then $H$ is a $(0,\tau)$-regular set of the Cayley graph $\Cay(G,X)$
where $X=Y\setminus H$ and $\tau=\frac{\vert H\vert\vert Y\vert}{|G|}$.

It is a natural question when a normal subgroup of a group is a regular set. This question was studied by Wang et al in \cite{WXZ2022+}. They proved that, for any finite group $G$, if a non-trivial normal subgroup $H$ of $G$ is a perfect code of $G$, then for any pair of integers $\kappa$ and $\tau$ with $0\leq\kappa\leq|H|-1$, $1\leq\tau\leq|H|$ and  $\gcd(2,|H|-1)\mid\kappa$, $H$ is also a $(\kappa,\tau)$-regular set of $G$. It was also shown in \cite{WXZ2022+} that there exists normal subgroups of some group which are $(\kappa,\tau)$-regular sets for some pair of integers $\kappa$ and $\tau$  but not perfect codes of the group. In this paper, we extend the mains results in \cite{WXZ2022+} by proving the following theorem.
\begin{theorem}
\label{main}
Let $G$ be a group and $H$ a non-trivial normal subgroup of $G$. Let $\kappa$ and $\tau$ be a pair of integers satisfying $0\leq\kappa\leq|H|-1$, $1\leq\tau\leq|H|$ and $\gcd(2,|H|-1)\mid\kappa$. The following two statements hold:
\begin{enumerate}
  \item if $\tau$ is even, then $H$ is a $(\kappa,\tau)$-regular set of $G$;
  \item if $\tau$ is odd, then $H$ is a $(\kappa,\tau)$-regular set of $G$ if and only if it is a perfect code of $G$.
\end{enumerate}
\end{theorem}

It was proved in \cite[Theorem 2.2]{HXZ2018} that a normal subgroup $H$ of $G$ is a perfect code of $G$ if and only if
\begin{enumerate}
  \item[\#] for any $g\in G$ with $g^{2}\in H$, there exists $h\in H$ such that $(gh)^2=1$.
\end{enumerate}
Note that condition \# always holds if $H$ is of odd order or odd index \cite[Corollary 2.3]{HXZ2018}.
Therefore, Theorem \ref{main} has a direct corollary as follows.
\begin{coro}
\label{odd}
Let $G$ be a group and $H$ a non-trivial normal subgroup of $G$. If either $|H|$ or $|G/H|$ is odd, then $H$ is a $(\kappa,\tau)$-regular set of $G$ for every pair of integers $\kappa$ and $\tau$ satisfying $0\leq\kappa\leq|H|-1$, $1\leq\tau\leq|H|$ and $\gcd(2,|H|-1)\mid\kappa$.
\end{coro}
\begin{remark*}
It is a challenging question whether Theorem \ref{main} and Corollary \ref{odd} can be generalized to non-normal subgroups $H$ of $G$.
\end{remark*}
\begin{remark*}
Let $H$ be a nontrivial normal subgroup of $G$ of even order not satisfying condition \#. Let $\kappa$ and $\tau$ be a pair of integers satisfying $0\leq\kappa\leq|H|-1$, $2\leq\tau\leq|H|$ and  $2\mid\tau$. Then Theorem \ref{main} (i) and \cite[Theorem 2.2]{HXZ2018} imply that $H$ is a $(\kappa,\tau)$-regular set but not a perfect code of $G$.
\end{remark*}
\section{Proof of Theorem \ref{main}}
Throughout this section, we use $\dot\cup_{i=1}^{n}S_{i}$ to denote the union of the pair-wise disjoint sets $S_1,S_2,\ldots,S_n$. Let $G$ be a group and $H$ a non-trivial normal subgroup of $G$. Let $\kappa$ and $\tau$ be a pair of integers satisfying $0\leq\kappa\leq|H|-1$, $1\leq\tau\leq|H|$ and $\gcd(2,|H|-1)\mid\kappa$. We firstly prove three lemmas and then complete the proof of Theorem \ref{main}.

\begin{lem}
\label{eventau}
If $\tau$ is even, then $H$ is a $(0,\tau)$-regular set of $G$.
\end{lem}
\begin{proof}
Let $A:=\{1,a_{1},\ldots,a_{s}\}$ be a left transversal of $H$ in $G$. Assume that the number of involutions contained in $a_{i}H$ is $n_{i}$ for every $1\leq i\leq s$. Let $\sigma$ be a permutation on $\{1,\ldots,s\}$ such that $a_{i}^{-1}H=a_{\sigma(i)}H$. Since $H$ is normal in $G$, we have
\begin{equation*}
a_{\sigma^{2}(i)}H=a_{\sigma(i)}^{-1}H=Ha_{\sigma(i)}^{-1}
=(a_{\sigma(i)}H)^{-1}=(a_{i}^{-1}H)^{-1}=Ha_{i}=a_{i}H.
\end{equation*}
It follows that $\sigma$ is the identity permutation or an involution. Assume that $\sigma$ fixes $t$ integers in $\{1,\ldots,s\}$. Then $0\leq t\leq s$ and $s-t$ is even. Set $\ell:=\frac{s-t}{2}$. Without loss of generality, we assume that
\begin{equation*}
\sigma(i)=\left\{
                   \begin{array}{ll}
                     i, & \hbox{if}~i\leq t; \\
                     i+\ell, & \hbox{if}~t<i\leq t+\ell;\\
                     i-\ell, & \hbox{if}~t+\ell<i\leq s.\\
                   \end{array}
                 \right.
\end{equation*}
Then $a_iH$ is inverse closed if $i\leq t$ and $(a_{t+j}H)^{-1}=a_{t+j+\ell}H$ for every positive integer $j$ not greater than $\ell$. In particular, $n_{i}=0$ if $i>t$.
For every $i\in\{1,\ldots,s\}$, take a subset $X_i$ of $a_iH$ of cardinality $\tau$ by the following rules:
\begin{itemize}
  \item if $i\leq t$ and $n_i\geq \tau$, then $X_i$ consists of exactly $\tau$ involutions;
  \item if $i\leq t$, $n_i<\tau$ and $\tau-n_i$ is even, then $X_i$ consists of $n_i$ involutions and $\frac{\tau-n_i}{2}$ pairs of mutually inverse elements of order grater than $2$;
 \item if $i\leq t$, $n_i<\tau$ and $\tau-n_i$ is odd, then $X_i$ consists of $n_i-1$ involutions and $\frac{\tau+1-n_i}{2}$ pairs of mutually inverse elements of order grater than $2$;
  \item if $t<i\leq t+\ell$, then $X_i$ consists of exactly $\tau$ elements of order greater than $2$;
  \item if $i> t+\ell$, then set $X_i=X_{i-\ell}^{-1}$.
\end{itemize}
Note that $X_1,\ldots,X_s$ are pair-wise disjoint. Set $X=\dot\cup_{i=1}^{s}X_{i}$. Then $X$ is an inverse closed subset of $G$ satisfying $X\cap H=\emptyset$ and $\vert X\cap gH\vert=\tau$ for every $g\in G\setminus H$. It follows that $H$ is a $(0,\tau)$-regular set of the Cayley graph $\Cay(G,X)$ and therefore a $(0,\tau)$-regular set of $G$.
\end{proof}
\begin{lem}
\label{oddtau}
If $\tau$ is odd, then $H$ is a $(0,\tau)$-regular set of $G$ if and only if it is a perfect code of $G$.
\end{lem}
\begin{proof}
The sufficiency follows from \cite[Theorem 1.2]{WXZ2022+}. Now we prove the necessity. Let $H$ be a $(0,\tau)$-regular set of the Cayley graph $\Cay(G,X)$. Then $X=X^{-1}$,
$X\cap H=\emptyset$ and $\vert X\cap gH\vert=\tau$ for every $g\in G\setminus H$. Let $A:=\{1,a_{1},\ldots,a_{s}\}$ be a left transversal of $H$ in $G$ and set $X_i=X\cap a_iH$ for every $i\in\{1,2,\ldots,s\}$. Then $X=\dot\cup_{i=1}^{s}X_{i}$. If $X_{i}$ contains an involution for each $i\in\{1,\ldots,s\}$, then $H$ is a perfect code of $G$ with respect to $\{1,y_1,\ldots,y_s\}$ where $y_i$ is an involution in $X_i$, $i=1,\ldots,s$. Now we assume that there exists at least one integer $k\in \{1,\ldots,s\}$ such that $X_{k}$ contains no involution. Then $x^{-1}\neq x$ for every element $x\in X_{k}$. It follows that $|X_{k}\cup X_{k}^{-1}|$ is even. Since $|X_{k}|=\tau$ and $\tau$ is odd, we get $X_k\neq X_k^{-1}$. Since $H$ is normal in $G$, we obtain $(a_kH)^{-1}=(Ha_k)^{-1}=a_k^{-1}H$. Assume that $a_k^{-1}H=a_{j}H$ for some $j\in \{1,\ldots,s\}$. Then $X_{k}^{-1}\subseteq a_{j}H$.
Since $X=\dot\cup_{i=1}^{s}X_{i}$ and $X^{-1}=X$, we conclude that $X_k^{-1}=X_{j}$. Therefore, without loss of generality, we can assume that $X_i^{-1}=X_{i+\ell}$ if $1\leq i\leq\ell$ and $X_i^{-1}=X_{i}$ if $2\ell< i\leq s$ where $\ell$ is a positive integer not greater than $\frac{s}{2}$. Note that $X_{i}$ contains at least one involution if $X_i^{-1}=X_{i}$ (as it is of odd cardinality).
For every $i\in\{1,\ldots,s\}$, take an element $y_i\in X_i$ by the following rules:
\begin{itemize}
  \item $y_i$ is an arbitrary element in $X_i$ if $i\leq \ell$;
  \item $y_i=y_{i-\ell}^{-1}$ if $\ell<i\leq 2\ell$;
  \item $y_i$ is an involution if $i>2\ell$.
\end{itemize}
Then $H$ is a perfect code of $G$ with respect to $\{1,y_1,\ldots,y_s\}$.
\end{proof}

\begin{lem}
\label{0kappa}
$H$ is a $(\kappa,\tau)$-regular set of $G$ if and only if $H$ is a $(0,\tau)$-regular set of $G$.
\end{lem}
\begin{proof}
$\Rightarrow$) Let $H$ be a $(\kappa,\tau)$-regular set of the Cayley graph $\Cay(G,X)$. Then $|H\cap X|=\kappa$ and $|gH\cap X|=\tau$ for every $g\in G\setminus H$. Set $Y=X\setminus H$. Then $|H\cap Y|=0$ and $|gH\cap Y|=\tau$ for every $g\in G\setminus H$.  Since $X^{-1}=X$ and $H^{-1}=H$, we get $Y^{-1}=Y$. It follows that $H$ is a $(0,\tau)$-regular set of the Cayley graph $\Cay(G,Y)$ and therefore a $(0,\tau)$-regular set of $G$.

$\Leftarrow$) Let $H$ be a $(0,\tau)$-regular set of the Cayley graph $\Cay(G,Y)$. Then $|H\cap Y|=0$ and $|gH\cap Y|=\tau$ for every $g\in G\setminus H$. Let $m$ be the number of elements contained in $H$ of order greater than $2$. Then $m$ is even and the number of involutions contained in $H$ is $|H|-1-m$. Recall that $0\leq\kappa\leq|H|-1$ and $\gcd(2,|H|-1)\mid\kappa$. If $\kappa$ is odd, then $|H|$ is even and therefore contains at least one involution. Take an inverse closed subset $Z$ of $H$ of cardinality $\kappa$ by the following rules:
\begin{itemize}
  \item if $m\geq\kappa$ and $\kappa$ is even, then $Z$ consists of exactly $\frac{\kappa}{2}$ pairs of mutually inverse elements of order grater than $2$;
  \item if $m\geq\kappa$ and $\kappa$ is odd, then $Z$ consists of $\frac{\kappa-1}{2}$ pairs of mutually inverse elements of order grater than $2$ and one involution;
  \item if $m<\kappa$, then $Z$ consists of $\frac{m}{2}$ pairs of mutually inverse elements of order grater than $2$ and $\kappa-m$ involutions.
\end{itemize}
Set $X=Y\cup Z$. Then $|H\cap X|=\kappa$ and $|gH\cap X|=\tau$ for every $g\in G\setminus H$.  Therefore $H$ is a $(\kappa,\tau)$-regular set of the Cayley graph $\Cay(G,X)$ and therefore a $(\kappa,\tau)$-regular set of $G$.
\end{proof}
\begin{proof}[Proof of Theorem \ref{main}] Lemma \ref{eventau} and Lemma \ref{0kappa} imply that $H$ is a $(\kappa,\tau)$-regular set of $G$ if $\tau$ is even. Now assume $\tau$ is odd. By Lemma \ref{oddtau} and Lemma \ref{0kappa}, $H$ is a $(\kappa,\tau)$-regular set of $G$ if and only if it is a perfect code of $G$.
\end{proof}

\medskip
\noindent {\textbf{Acknowledgements}}~~
The first author was supported by Natural Science Foundation of Chongqing (cstc2019jcyj-msxmX0146) and the Foundation of Chongqing Normal University (21XLB006).

{\small \baselineskip=10pt
\begin{thebibliography}{99}
\bibitem{AC2013}
 M. An\dj eli\'c, D.M. Cardoso, S.K. Simi\'c, Relations between $(\kappa,\tau)$-regular sets and star complements, \emph{ Czechoslovak Math. J.} 63 (138) (2013) 73--90.
\bibitem{Car2019}  D. M. Cardoso, An overview of $(\kappa,\tau)$-regular sets and their applications, \emph{Discrete Appl. Math.} 269 (2019) 2--10.
\bibitem{CR2004}
D.M. Cardoso, P. Rama, Equitable bipartitions of graphs and related results, \emph{J. Math. Sci.} 120 (1) (2004) 869--880.
\bibitem{CR2007}
D.M. Cardoso, P. Rama, Spectral results on graphs with regularity constraints, \emph{Linear Algebra Appl.} 423 (2007) 90--98.
\bibitem{CS2010}
D.M. Cardoso, I. Sciriha, C. Zerafa, Main eigenvalues and $(\kappa,\tau)$-regular sets, \emph{Linear Algebra Appl.} 423 (2010) 2399--2408.
\bibitem{GL2020}
H.M. Green, M.W. Liebeck, Some codes in symmetric and linear groups, \emph{ Discrete Math.} 343(8) (2020) 111719.

\bibitem{HXZ2018} H. Huang,  B.Z. Xia,  S.M. Zhou, Perfect codes in Cayley graphs, \emph{SIAM J. Discrete Math.} 32 (2018) 548--559.

\bibitem{WXZ2022} Y. Wang, B.Z. Xia,  S.M. Zhou, Subgroup regular sets in Cayley graphs, \emph{ Discrete Math.}  345(11) (2022) 113023.
\bibitem{WXZ2022+}
Y. Wang, B.Z. Xia, S.M. Zhou, Regular sets in Cayley graphs, \emph{J. Algebr. Comb.} (2022) https://doi.org/10.1007/s10801-022-01181-8.
\end{thebibliography}}
\end{document}